\documentclass[12pt]{amsart}

\usepackage{amsmath}
\usepackage{amsfonts}
\usepackage{amsthm}

\usepackage[cmtip,arrow]{xy}
\usepackage{pb-diagram,pb-xy}

\newtheorem{theorem}{Theorem}
\newtheorem{corollary}{Corollary}
\newtheorem{lemma}{Lemma}
\newtheorem{proposition}{Proposition}
\newtheorem{definition}{Definition}
\newtheorem{example}{Example}

\begin{document}
\title{Gorenstein injective precovers, covers, and envelopes}
\author{E. Enochs and S. Estrada and A. Iacob}

\begin{abstract}
We give a sufficient condition for the class of Gorenstein injective modules be precovering: if $R$ is right noetherian and if the class of Gorenstein injective modules, $\mathcal{GI}$, is closed under filtrations, then $\mathcal{GI}$ is precovering in $R-Mod$. The converse is also true when we assume that $\mathcal{GI}$ is covering.
 We extend our results to the category of complexes. We prove that if the class of Gorenstein injective modules is closed under filtrations then the class of Gorenstein injective complexes is precovering in $Ch(R)$. We also give a sufficient condition for the existence of Gorenstein injective covers. We prove that if the ring $R$ is commutative noetherian and such that the character modules of Gorenstein injective modules are Gorenstein flat, then the class of Gorenstein injective complexes is covering. And we prove that over such rings every complex also has a Gorenstein injective envelope. In particular this is the case when the ring is commutative noetherian with a dualizing complex.\\
The second part of the paper deals with Gorenstein projective and flat complexes. We prove that over commutative noetherian rings of finite Krull dimension every complex of $R$-modules has a \textbf{special} Gorenstein projective precover.
\end{abstract}

\maketitle

\section{Introduction}
The Gorenstein injective modules were introduced by Enochs and Jenda (in \cite{enochs:95:gorenstein}) as a generalization of the injective modules.
Together with the Gorenstein projective and Gorenstein flat modules, they are the key ingredients of Gorenstein homological algebra. This is the reason why the existence of the Gorenstein resolutions has been studied extensively in the past years.\\
We consider the existence of Gorenstein injective precovers and covers.\\
 It is known that a precovering class that is closed under direct summands is
also closed under arbitrary direct sums. So when a ring $R$ is such that every
$R$-module has a Gorenstein injective precover, every direct sum of Gorenstein
injective modules is still Gorenstein injective. By
\cite{christensen:11:beyond} Proposition 3.15, such a ring $R$ must be
noetherian. It is not known whether the converse holds.\\
We give a sufficient condition for the class $\mathcal{GI}$ of Gorenstein injective modules be precovering: we prove that if $R$ is right noetherian and if the class of Gorenstein injective modules is closed under filtrations, then $\mathcal{GI}$ is precovering in $R-Mod$. The converse is also true when we assume the existence of special Gorenstein injective precovers. In particular, this is the case when the class $\mathcal{GI}$ is covering.\\
We extend our results to the category $Ch(R)$ of complexes of $R$-modules over a two sided noetherian ring $R$. We prove that when $\mathcal{GI}$ is closed under filtrations the class of Gorenstein injective complexes is precovering.\\
We also show that when $\mathcal{GI}$  is closed under direct limits (and so $\mathcal{GI}$  is covering in $R-Mod$), the class of Gorenstein injective complexes is covering in $Ch(R)$. We prove that this is the case when $R$ is commutative noetherian and such that for every Gorenstein injective module $M$, its character module $M^+$ is Gorenstein flat.\\
Then we consider the existence of Gorenstein injective envelopes for complexes. We prove (Theorem 5) that if $R$ is a commutative noetherian ring with the property that the character modules of Gorenstein injectives are Gorenstein flat, then every complex has a Gorenstein injective envelope. In particular, this is the case when $R$ is commutative noetherian with a dualizing complex.


As already noted, the Gorenstein injective, projective and flat modules are the key elements of Gorenstein homological algebra. Enochs and L\'opez-Ramos proved the existence of Gorenstein flat covers over coherent rings. And J{\o}rgensen showed the existence of Gorenstein projective precovers over commutative noetherian rings with dualizing complexes. More recently, Murfet and Salarian extended his result to commutative noetherian rings of finite Krull dimension.\\
We extend some of the results on the existence of Gorenstein projective and Gorenstein flat (pre)covers to the category of complexes of $R$-modules over  noetherian rings. We show the existence of Gorenstein flat covers over two sided noetherian rings. And we prove the existence of \textbf{special} Gorenstein projective precovers over commutative noetherian rings of finite Krull dimension. This apparently "slightly" variation is crucial from a homotopical point of view, since it allows to define Gorenstein projective cofibrant replacements of modules in the category of unbounded complexes.

\vspace{7mm}

\section{ Gorenstein injective precovers and covers for modules and for complexes}

Throughout the paper $R$ will denote an associative ring with 1. By $R$-module we mean left $R$-module.\\
We recall that a module $G$ is Gorenstein injective if there is an exact and $Hom
(Inj, -)$ exact complex $ \ldots \rightarrow E_1 \rightarrow E_0
\rightarrow E_{-1} \rightarrow E_{-2}\rightarrow \ldots $ of
injective modules such that $G = Ker(E_0 \rightarrow E_{-1})$.

The notions of precover and cover, preenvelope and envelope with
respect to a class of modules $\mathcal{C}$ were introduced by
Enochs in \cite{enochs:81:existence}. We are interested here in
Gorenstein injective precovers and covers.

\begin{definition}
 A morphism $\phi: E\rightarrow X$ is a
Gorenstein injective precover of $X$ if $E$ is Gorenstein
injective and if $Hom(F, E) \rightarrow Hom(F, X)$ is surjective
for all Gorenstein injective modules $F$. If moreover, any $f:
E\rightarrow E$ such that $\phi \circ f=\phi$ is an automorphism of
$E$ then $\phi: E\rightarrow X$ is called a Gorenstein injective
cover of $X$.
\end{definition}

Gorenstein injective preenvelopes and envelopes are defined dually.

Over Gorenstein rings the existence of Gorenstein injective covers is known
(see for instance, \cite[Theorem 11.1.3]{enochs:00:relative}). We give
sufficient conditions in order for the existence of Gorenstein injective
precovers and covers over noetherian rings.\\



We recall the following\\

\begin{definition}  A direct system of modules ($X_\alpha|\alpha \le \lambda$) is said to be
continuous if $X_0 = 0$ and if for each limit ordinal $\beta \le
\lambda$ we have $X_\beta = \underrightarrow{lim} X_\alpha$ with the
limit over the $\alpha < \beta$. The direct system $(X_\alpha|
\alpha \le \lambda$) is said to be a system of monomorphisms if all
the morphisms in the system are monomorphisms.
\end{definition}

If ($X_\alpha|\alpha \le \lambda$) is a continuous direct system of
$R$-modules then for this to be a system of monomorphisms it
suffices that $X_\alpha \rightarrow X_{\alpha+1}$ be monomorphism
whenever $\alpha + 1 \le \lambda$. 

\begin{definition} Let $\mathcal{L}$ be a class of $R$-modules.
An $R$-module X of A is said to be a direct transfinite extension of
objects of $\mathcal{L}$ if $X = \underrightarrow{lim} X_\alpha$ for
a continuous direct system ($X_\alpha|\alpha \le \lambda$) of
monomorphisms such that $coker (X_\alpha \rightarrow X_{\alpha+1}) $
is in $\mathcal{L}$ whenever $\alpha + 1 \le \lambda$.
\end{definition}

\begin{definition}
By a filtration of a module $M$ we mean that for some ordinal number $\lambda$ we have a continuous well-ordered chain $(M_{\alpha} | \alpha \le \lambda)$ of submodules of $M$ with $M_0 = 0$ and with $M_{\lambda} = M$. We say that $\lambda$ is the length of the filtration. If $\mathcal{C}$ is any class of modules, this filtration is said to be a $\mathcal{C}$-filtration if for every $\alpha + 1 \le \lambda$ we have that $M_{\alpha + 1}/M_{\alpha}$ is isomorphic to some $C \in \mathcal{C}$.\\
The class of all $\mathcal{C}$-filtered modules is denoted $Filt(\mathcal{C})$.
\end{definition}

Roughly speaking, $Filt(\mathcal{C})$ is the class of all transfinite extensions of modules in $\mathcal{C}$.

It is known (\cite{enochs:11:filt}, Theorem 5.5, and \cite{stovicek:11:deconstructibility}, Theorem in the Introduction) that if $\mathcal{C}$ is a \textbf{set} of modules, then $Filt(\mathcal{C})$ is precovering.

Our first result is a sufficient condition for the existence of Gorenstein injective precovers. 
 It is known (\cite{enochs:02:kaplansky}) that when $R$ is right noetherian the class of Gorenstein injective left $R$-modules is a Kaplansky class.\\ Since we use this property in our proof, we recall the definition:\\


\begin{definition}
Let $R$ be a ring, $\kappa$ an infinite cardinal, and $\mathcal{A}$
a class of $R$-modules. $\mathcal{A}$ is $\kappa$-Kaplansky if for
each $0 \neq A \in \mathcal{A}$ and each $\kappa$-generated
submodule $B \subseteq A$ there exists a $\kappa$-presentable
submodule $C \in \mathcal{A}$ such that $B \subseteq C \subseteq A$
and $A/C \in \mathcal{A}$.\\
\end{definition}


\begin{lemma}\label{GIKap}
Let $R$ be a right noetherian ring. There exists an infinite regular cardinal
$\kappa$ such that the class $\mathcal{GI}$ of Gorenstein injective left
$R$-modules is a $\kappa$--Kaplansky class.
\end{lemma}
\begin{proof}
This is \cite[Proposition 2.6]{enochs:02:kaplansky}.
\end{proof}

\begin{proposition}\label{ppGI}
Let $R$ be a right noetherian ring, $\kappa$ an infinite regular cardinal as
in Lemma \ref{GIKap}, and let $\mathcal{X}$ denote a set of
representatives of isomorphism classes of Gorenstein injective
modules $M$ such that $|M|\leq \kappa$. The following assertions are equivalent:
\begin{enumerate}
 \item $\mathcal{GI}$ is closed under $\mathcal{X}$--filtrations.
\item $\mathcal{GI}=Filt(\mathcal{X})$.
\end{enumerate}

\end{proposition}

\begin{proof}

$(2)\Rightarrow (1)$ Clear.\medskip \par\noindent

$(1)\Rightarrow (2)$ By $(1)$ it is clear that $Filt(\mathcal{X})\subseteq
\mathcal{GI}$. Conversely, let $G \neq 0$ be a Gorenstein injective
module and let $\{g_\alpha, \alpha < \lambda \}$ be a generating set
of $G$. Let $G_0 =0$; if $G_\alpha$ is defined such that $G_\alpha$
is Gorenstein injective then let $A=G/G_\alpha$ and $B=(g_\alpha +
G_\alpha)R$. Since $\mathcal{GI}$ is Kaplansky there exists
$G_{\alpha +1} \subseteq G$ such that $G_\alpha \bigcup \{g_\alpha
\} \subseteq G_{\alpha +1}$ and $G_{\alpha +1} / G_\alpha$ is
Gorenstein injective. Then $G/G_{\alpha+1} \simeq
\frac{G/G_\alpha}{G_{\alpha+1}/G_\alpha}$ is Gorenstein injective
and $G_{\alpha+1}$ is Gorenstein injective because $\mathcal{GI}$ is
closed under extensions. If $\beta \le \lambda$ is a limit ordinal,
then we set $G_\beta = \bigcup_{\alpha < \beta} G_\alpha$. From $(1)$ we infer
that $G_{\beta}\in \mathcal{GI}$. Now,
since the sequence $0 \rightarrow G_\beta \rightarrow G \rightarrow
G/G_\beta \rightarrow 0$ is exact with $G_\beta$ and $G$ Gorenstein
injectives, and $\mathcal{GI}$ is closed under cokernels of
monomorphisms, it follows that $G/G_\beta$ is Gorenstein injective,
so we can continue the induction. The process clearly terminates.
 Thus $G\in Filt(\mathcal{X})$.

\end{proof}

\begin{theorem}
Under the assumptions of Proposition \ref{ppGI} the class
of Gorenstein injective modules is
precovering.
\end{theorem}

\begin{proof}
By
\cite{enochs:11:filt}, Theorem 5.5, or by \cite{stovicek:11:deconstructibility}, Theorem in the Introduction, $Filt(\mathcal{\mathcal{X}})$ is
precovering.
\end{proof}

If moreover every $R$-module has a special $\mathcal{GI}$-precover then the converse is also true. As noted in the introduction, the class $\mathcal{GI}$ being precovering requires $R$ to be left noetherian.\\

We show that in this case the class $\mathcal{GI}$ is closed under transfinite extensions.





\begin{proposition}\label{closureTE}
Let $R$ be a two sided noetherian ring. If every $R$-module has a special Gorenstein injective precover then
the class $\mathcal{GI}$ of
Gorenstein injective modules is closed under
direct transfinite extensions.
\end{proposition}

\begin{proof}
Let ($G_\alpha|\alpha \le \lambda$) be a direct system of
monomorphisms, with each $G_\alpha \in \mathcal{GI}$, and let $G=
\underrightarrow{lim} G_\alpha$.  Since for each $\alpha$ we have\\
$G_\alpha \in ^\bot(\mathcal{GI}^\bot)$ it follows that $G=
\underrightarrow{lim} G_\alpha \in ^\bot(\mathcal{GI}^\bot)$ (by \cite{eklof:97:homological}, Theorem 1.2).\\
 For
each $\alpha$ consider $\oplus_{E \in X} E
^{(Hom(E,G_\alpha))} \rightarrow G_\alpha$ where the map is the
evaluation map, and $X$ is a representative set of indecomposable
injective modules $E$. This is an injective precover of $G_\alpha$,
and since $G_\alpha$ is Gorenstein injective, $\oplus_{E \in X}
E ^{(Hom(E,G_\alpha))} \rightarrow G_\alpha$ is surjective.
Also this way of constructing a precover is functorial. The map
$G_\alpha \rightarrow G_\beta$ gives rise to a a map $E_\alpha
\rightarrow E_\beta$. Since $E_\alpha \rightarrow G_\alpha$ was
constructed in a functorial manner, we have that when $\alpha \le
\beta \le \gamma$ the map $E_\alpha \rightarrow E_\gamma$ is the
composition of the two maps $E_\alpha \rightarrow E_\beta$ and
$E_\beta \rightarrow E_\gamma$.

 Then we have an exact sequence
$E \rightarrow G \rightarrow 0$ with $E = \underrightarrow{lim}
E_\alpha$ an injective module. It follows that $G$ has a surjective
injective cover and therefore a surjective special Gorenstein
injective precover. So there is an exact sequence
$$0 \rightarrow A \rightarrow \overline{G} \rightarrow G \rightarrow 0$$ with
$A \in \mathcal{GI}^\bot$ and $\overline{G}$ Gorenstein injective. But by the
above we have that $Ext^1 (G,A)=0$. So $G$ is a direct summand of $\overline{G}
\in \mathcal{GI}$.
\end{proof}

\begin{corollary}
Let $R$ be a two sided noetherian. If the class of Gorenstein injective modules, $\mathcal{GI}$, is covering then $\mathcal{GI}$ is closed under transfinite extensions.
\end{corollary}

\begin{proposition}
Let $R$ be two sided noetherian. If every $R$-module has a special Gorenstein injective precover then $\mathcal{GI}$ is closed under $\mathcal{X}$-filtrations.
\end{proposition}

\begin{proof}
By Proposition 2 above $Filt (\mathcal{X}) \subseteq \mathcal{GI}$.
\end{proof}

\begin{proposition}
When the ring $R$ is two sided noetherian and the class of Gorenstein injective modules is closed under direct limits, the class $\mathcal{GI}$ is covering.
\end{proposition}

\begin{proof}
Since $R$ is right noetherian the class of Gorenstein injective modules is
Kaplansky. Since $\mathcal{GI}$ is also closed under direct limits,
it is precovering (by Theorem 1).
A precovering class that is also closed under direct limits is
covering (\cite{enochs:00:relative}, Corollary, 5.2.7)
\end{proof}

\begin{corollary}(\cite{enochs:00:relative}, Theorem 11.1.3)
Over a Gorenstein ring the class of Gorenstein injective modules is
covering.
\end{corollary}



\begin{corollary} (\cite{holm:09:cotorsion.pais}, Theorem 3.3)
If $R$ is commutative noetherian with a dualizing complex then the class of Gorenstein injective modules is covering.
\end{corollary}

\begin{proof}
By \cite{christensen:06:gorenstein}, Theorem 6.9, $\mathcal{GI}$ is closed under direct limits. By Proposition 4, it is covering.
\end{proof}

We extend our results to the category of complexes of $R$-modules over a two sided noetherian ring $R$. We will use the notation $\mathcal{GI}(C)$ for the class of Gorenstein injective complexes.

It is known that when $R$ is a left noetherian ring, a complex of
left $R$-modules is Gorenstein injective if and only if each component is
a Gorenstein injective $R$-module (\cite{liu:09:gorinj}, Theorem 8). Using
this result we prove:\\

\begin{proposition}
Let $R$ be a two sided noetherian ring. If the class of Gorenstein injective
$R$-modules is closed under filtrations then the class of Gorenstein
injective complexes is precovering in $Ch(R)$.
\end{proposition}

\begin{proof}
Let again $\kappa$ be an infinite regular cardinal as
in Lemma \ref{GIKap}, and let $\mathcal{X}$ denote a set of
representatives of isomorphism classes of Gorenstein injective
modules $M$ such that $|M|\leq \kappa$. Then $\mathcal{GI} = Filt(\mathcal{X})$.
Since $\mathcal{GI}$ is closed under filtrations, it follows that each
 complex of Gorenstein injective modules is filtered by bounded below complexes
with components in $\mathcal{X}$ (\cite{stovicek:11:deconstructibility}, Proposition
4.3). In particular, the class of complexes of Gorenstein injective modules is deconstructible. By \cite{liu:09:gorinj} Theorem 8, this is the class of Gorenstein
injective complexes. By \cite{stovicek:11:deconstructibility},
Theorem (page 2), the class of Gorenstein injective complexes is
precovering.
\end{proof}




\begin{proposition}
Let $R$ be two sided noetherian. If the class of Gorenstein injective modules
is closed under direct limits, then the class of Gorenstein
injective complexes is covering in $Ch(R)$.
\end{proposition}

\begin{proof}
By Proposition 4 the class of Gorenstein injective modules is covering. Also, the class of Gorenstein injective modules is closed under extensions, direct products, and by our assumptions, closed under direct limits. By \cite{enochs:11:gorenstein complexes}, Theorem 1, the class of complexes of Gorenstein
injective modules is covering. By \cite{liu:09:gorinj}, Theorem 8,
these are the Gorenstein injective complexes.
\end{proof}

We give a sufficient condition for the class of Gorenstein injective complexes be covering.
Since it involves Gorenstein flat modules, we recall the following\\


\begin{definition}
A module $N$ is Gorenstein flat if there is an exact and $Inj
\otimes -$ exact sequence $ \ldots \rightarrow F_1 \rightarrow F_0
\rightarrow F_{-1} \rightarrow F_{-2}\rightarrow \ldots $ of flat
modules such that $N = Ker(F_0 \rightarrow F_{-1})$.
\end{definition}

In \cite{enochs:12:gorenstein.injective.covers}, Theorem 2, we proved the following result: when the ring $R$ is commutative noetherian and with the property that the character modules of the Gorenstein injective modules are Gorenstein flat, the class of Gorenstein injective modules is closed under direct limits, and so it is covering in $R-Mod$.\\

By Proposition 6 and \cite{enochs:12:gorenstein.injective.covers}, Theorem 2, we have :\\
\begin{theorem}
Let $R$ be a commutative noetherian ring. Assume that the character modules of Gorenstein injective modules are Gorenstein flat. Then the class of Gorenstein injective complexes is covering.
\end{theorem}

\begin{example}
If the ring $R$ is commutative noetherian with a dualizing complex then the class of Gorenstin injective complexes is covering.
\end{example}

\textbf{Gorenstein injective envelopes of complexes }

In \cite{enochs:12:gorenstein.injective.covers} we proved that the class of Gorenstein injective modules is enveloping over a commutative noetherian ring with the property that the character modules of Gorenstein injective modules are Gorenstein flat. In particular this shows the existence of Gorenstein injective envelopes over commutative noetherian rings with dualizing complexes.

We extend this result to the category of complexes. We will denote by $\mathcal{GI}(C)$ the class of Gorenstein injective complexes, and by $\mathcal{GF}(C)$ the class of Gorenstein flat complexes.\\
 We recall the definition of a Gorenstein flat complex. The definition is given
(\cite{garcia:99:covers}) in terms
of the tensor product of complexes.\\

We recall that if $C$ is a complex of right $R$-modules and $D$ is a
complex of left $R$-modules then the usual tensor product complex of
$C$ and $D$ is the complex of $Z$-modules $C \otimes ^. D$ with $(C
\otimes ^. D)_n = \oplus_{t \in Z} (C_t \otimes_R D_{n-t})$ and
differentials $$\delta (x \otimes y) = \delta^C _t (x) \otimes y +
(-1)^t x
\otimes \delta^D _{n-t}(y)$$ for $x \in C_t$ and $y \in D_{n-t}$.\\

In \cite{garcia:99:covers},  Garc\'{\i}a Rozas introduced another
tensor product: if $C$ is again a complex of right $R$-modules and
$D$ is a complex of left $R$-modules then $C \otimes D$ is defined
to be $\frac{C \otimes ^. D}{B(C \otimes ^. D)}$. Then with the maps
$$\frac{(C \otimes ^. D)_n}{B_n(C \otimes ^. D)} \rightarrow \frac{(C
\otimes ^. D)_{n-1}}{B_{n-1}(C \otimes ^. D)}$$ $x \otimes y
\rightarrow \delta_C (x) \otimes y$, where $x \otimes y$ is used to
denote the coset in $\frac{C \otimes ^. D}{B(C \otimes ^. D)}$ we
get a complex.\\ This is the tensor product used to define
Gorenstein flat complexes.\\

We recall that a complex $F$ is flat if $F$ is exact and if for each
$i \in Z$ the module $Ker (F_i \rightarrow F_{i-1})$ is flat
(\cite{garcia:99:covers}, Theorem 4.1.3).

\begin{definition}(\cite{garcia:99:covers})
A complex $X$ of left $R$-modules is Gorenstein flat if there exists
an exact sequence $$\ldots \rightarrow F_1 \rightarrow F_0
\rightarrow F_{-1} \rightarrow F_{-2}\rightarrow \ldots $$  such
that \\
1) each $F_i$ is a flat complex\\
2) $C= Ker (F_0 \rightarrow F_{-1})$\\
3) the sequence remains exact when $E \otimes -$ is applied for any
injective complex of right $R$-modules $E$.
\end{definition}

We prove first that if the ring $R$ is noetherian then the class of Gorenstein injective complexes is enveloping if and only if its left orthogonal class, $^\bot \mathcal{GI}(C)$ is covering.\\
We start with the following result:\\
\begin{proposition}
Let $R$ be a left noetherian ring. Then $(^\bot \mathcal{GI}(C), \mathcal{GI}(C))$ is a complete hereditary cotorsion pair in the category $Ch(R)$ of complexes of $R$-modules.
\end{proposition}

\begin{proof}

By \cite{gillespie:12:gorenstein complexes}, $(^\bot \mathcal{GI}(C), \mathcal{GI}(C))$ is a complete cotorsion pair whenever R is any (left) Noetherian ring.\\
It remains to show that the pair is hereditary. Let $L \in ^\bot \mathcal{GI}(C)$ and let $G$ be a Gorenstein injective complex. Then there exists an exact sequence of complexes $0 \rightarrow G \rightarrow E \rightarrow G' \rightarrow 0$ with $E$ an injective complex, and with $G' \in \mathcal{GI}(C)$. This gives an exact sequence $0=Ext^1(L,G') \rightarrow Ext^2(L,G) \rightarrow Ext^2(L,E) =0$. Thus $Ext^2(L,G)=0$. Similarly $Ext^i(L,G)=0$ for all $i \ge 1$.

\end{proof}

The following result is proved for modules in \cite{enochs:02:gor.flat.covers} (Theorem 1.4). The argument carries to the category of complexes:\\

\begin{theorem}(\cite{enochs:02:gor.flat.covers}, Theorem 1.4)
Let $(\mathcal{L}, \mathcal{C})$ be a hereditary cotorsion pair in $Ch(R)$. Then the following are equivalent:\\
(1) $(\mathcal{L}, \mathcal{C})$ is perfect.\\
(2) Every complex of $R$-modules has a $\mathcal{C}$ envelope and every $C \in \mathcal{C}$ has an $\mathcal{L}$-cover.\\
(3) Every complex of $R$-modules has an $\mathcal{L}$-cover and every $L \in \mathcal{L}$ has a $\mathcal{C}$-envelope.

\end{theorem}

We use this result to prove the following\\

\begin{theorem}
Let $R$ be a noetherian ring. The following are equivalent:\\
(1) The cotorsion pair $(^\bot \mathcal{GI}(C), \mathcal{GI})$ is perfect.\\
(2) The class $^\bot \mathcal{GI}(C)$ is covering.\\
(3) The class $\mathcal{GI}(C)$ is enveloping.
\end{theorem}

\begin{proof}
(1) $\Rightarrow$ (2) and (1) $\Rightarrow$ (3) are immediate from the definition.\\
(2) $\Rightarrow$ (1) By Theorem 3 it suffices to prove that every complex $L$ in $^\bot \mathcal{GI}(C)$ has a Gorenstein injective envelope.\\
Let $L \in ^\bot \mathcal{GI} (C)$, and let $L \rightarrow E$ be an injective envelope. Then
we have an exact sequence $0 \rightarrow L \rightarrow E \rightarrow Y \rightarrow 0$.
Also, by Proposition 7, there exists an exact sequence $0 \rightarrow L \rightarrow G \rightarrow C \rightarrow 0$ with $L \rightarrow G$ a Gorenstein injective preenvelope and with $C \in ^\bot \mathcal{GI}(C)$. But then $G \in ^\bot \mathcal{GI}(C) \cap \mathcal{GI}(C)$,  so $G$ is an injective complex. Since $L \rightarrow G$ is an injective preenvelope we have $G \simeq E \oplus E'$.\\
We have a commutative diagram:\\

\[
\begin{diagram}
\node{0}\arrow{e}\node{L}\arrow{s,=}\arrow{e}\node{E}\arrow{s}\arrow{e}\node{Y}\arrow{s}\arrow{e}\node{0}\\
\node{0}\arrow{e}\node{L}\arrow{e}\arrow{e}\node{G}\arrow{e}\node{C}\arrow{e}\node{0}
\end{diagram}
\]

This gives an exact sequence $0 \rightarrow E \rightarrow G \oplus Y \rightarrow C \rightarrow 0$. Since $E$ is injective we have $G \oplus Y \simeq  E \oplus C$ and therefore $C \simeq Y \oplus E'$. It follows that $Y \in ^\bot \mathcal{GI}(C)$.
So the sequence $0 \rightarrow L \xrightarrow{i} E \rightarrow Y \rightarrow 0$ is exact with $E$ Gorenstein injective and with $Y \in ^\bot \mathcal{GI}(C)$. Thus $L \rightarrow E$ is a Gorenstein injective preenvelope. Also, any $u:E \rightarrow E$ such that $ui =i$ is an automorphism of $E$ (because $L \rightarrow E$ is an injective envelope).\\
(3) $\Rightarrow$ (1) By Theorem 3, it suffices to show that every complex in $\mathcal{GI}(C)$ has a $^\bot \mathcal{GI}(C)$ cover. \\
Let $X \in \mathcal{GI}(C)$. Since $X \in \mathcal{GI}(C)$ there exists an exact sequence $0 \rightarrow G \rightarrow U \rightarrow X \rightarrow 0$ with $U$ injective and with $G$ Gorenstein injective. Consider also an exact sequence $0 \rightarrow I \rightarrow E \rightarrow X \rightarrow 0$ with $E \rightarrow X$ an injective cover. Then $E \in ^\bot \mathcal{GI}(C)$ and $I \in Inj^\bot$.\\
 Since $U \rightarrow X$ is an injective precover, we have $U \simeq E \oplus E'$. This gives that $G \simeq I \oplus I'$, so $I$ is Gorenstein injective. Thus the sequence $0 \rightarrow I \rightarrow E \xrightarrow{\phi} X \rightarrow 0$  is exact with $E \in ^\bot \mathcal{GI}(C)$ and with $I \in \mathcal{GI}(C)$. Then $E \rightarrow X$ is a $^\bot \mathcal{GI}(C)$ precover. Since any $v:E \rightarrow E$ such that

$\phi v = \phi$ is an automorphism of $E$, it follows that $E \xrightarrow{\phi} X$ is a $^\bot \mathcal{GI}(C)$ cover.

\end{proof}

For the following result we recall some definitions from \cite{garcia:99:covers}.\\
Given two complexes $C$ and $D$, let $\underline{Hom}(C,D) = Z (\mathcal{H}om(C,D)$. $\underline{Hom}(C,D))$ can be made into a complex with $\underline{Hom}(C,D)_m $ the abelian group of morphisms from $C$ to $D[m]$ and with a boundary operator given by: if $f \in \underline{Hom}(C,D)_m$ then $\delta_m(f):C \rightarrow D[m+1]$ with $\delta_m(f)_n = (-1)^m \delta_D f^n$, for any $n \in \mathbb{Z}$.\\
The right derived functors of $\underline{Hom}(C,D)$ are complexes denoted $\underline{Ext^i}(C,D)$.\\ $\underline{Ext^i}(C,D)= \ldots \rightarrow Ext^i(C, D[n+1]) \rightarrow Ext^i(C, D[n]) \rightarrow \ldots$, with boundary operator induced by the boundary operator of $D$.\\
The right derived functors of the tensor product $- \otimes -$ are denoted by $Tor_i(-,-)$.

We prove that a complex $K$ is in the left orthogonal class of $\mathcal{GI}(C)$ if and only if the complex $K^+$ is in the right orthogonal class of $\mathcal{GF}(C)$:\\
\begin{proposition}
Let $R$ be a commutative noetherian ring with the property that the character modules of Gorenstein injective modules are Gorenstein flat. Then a complex $K$ is in $^\bot \mathcal{GI}(C)$ if and only if $K^+ \in \mathcal{GF}(C)^\bot$.
\end{proposition}

\begin{proof}
"$\Rightarrow$" Let $B$ be a Gorenstein flat complex. Since $R$ is noetherian this is equivalent to $B_n \in \mathcal{GF}$ for all $n \in \mathbb{Z}$ (\cite{enochs:11:gorenstein complexes}, Lemma 12). Then $B^+$ is a complex of Gorenstein injective modules, so $B^+ \in \mathcal{GI}(C)$ (by \cite{liu:09:gorinj}, Theorem 8). So we have $\underline{Ext}^1(K,B^+)=0$. Then $\underline{Ext}^1(B,K^+) \simeq Tor_1(B,K)^+ \simeq \underline{Ext}^1(K,B^+) =0$. Thus $Ext^1(B, K^+) =0$ for any $B \in \mathcal{GF}$. It follows that $K^+ \in \mathcal{GF}(C)^\bot$.\\

"$\Leftarrow$" Assume that $K^+ \in \mathcal{GF}(C) ^\bot$. \\
Since $R$ is noetherian, by Proposition 7 there exists an exact sequence $0 \rightarrow K \rightarrow G \rightarrow V \rightarrow 0$ with $G$ Gorenstein injective and with $V$ in $^\bot \mathcal{GI}(C)$. Then we have an exact sequence $0 \rightarrow V^+ \rightarrow G^+ \rightarrow K^+ \rightarrow 0$. By the above $V^+ \in \mathcal{GF}(C)^\bot$. By our assumption, $K^+ \in \mathcal{GF}(C)^\bot$. It follows that $G^+ \in \mathcal{GF}(C)^\bot \cap \mathcal{GF}(C)$. But this means that $G^+$ is a flat complex, and therefore $G$ is an injective complex. So we have an exact sequence $0 \rightarrow K \rightarrow G \rightarrow V \rightarrow 0$ with both $G$ and $V$ in $^\bot \mathcal{GI}(C)$. Since the pair $(^\bot\mathcal{GI}(C), \mathcal{GI}(C))$ is hereditary it follows that $ K \in ^\bot \mathcal{GI}(C)$.
\end{proof}

We can prove now:\\
\begin{proposition}
Let $R$ be commutative noetherian and such that the character modules of Gorenstein injective modules are Gorenstein flat. Then the class $^\bot \mathcal{GI}(C)$ is closed under pure quotients.
\end{proposition}

\begin{proof}

Let $0 \rightarrow A \rightarrow B \rightarrow C \rightarrow 0$ be a pure exact sequence of complexes with $B \in ^\bot \mathcal{GI}(C)$. Then the sequence $0 \rightarrow C^+ \rightarrow B^+ \rightarrow A^+ \rightarrow 0$ is split exact. So $B^+ \simeq A^+ \oplus C^+$. By Proposition 8, the complex $B^+$ is in $\mathcal{GF}^\bot$. It follows that both $A^+$ and $C^+$ are in $\mathcal{GF}(C)^\bot$. By Proposition 8 again, $A$ and $C$ are both in $^\bot \mathcal{GI}(C)$.
\end{proof}



\begin{theorem}
Let $R$ be a commutative noetherian ring such that the character modules of Gorenstein injective modules are Gorenstein flat. Then the class of Gorenstein injective complexes is enveloping in $Ch(R)$.
\end{theorem}

\begin{proof}
By Theorem 4 it suffices to prove that the class $^\bot \mathcal{GI}(C)$ is covering. \\
By Proposition 7 the class $^\bot \mathcal{GI}(C)$ is precovering, so it is closed under direct sums. Since the direct limit of an inductive family is a pure quotient of the direct sum, by Proposition 9, every direct limit of complexes in $^\bot \mathcal{GI}(C)$ is still in $^\bot \mathcal{GI}(C)$. It follows that $^\bot \mathcal{GI}(C)$ is covering in $Ch(R)$.
\end{proof}

\begin{corollary}
If $R$ is a commutative noetherian ring with a dualizing complex, then every complex of $R$-modules has a Gorenstein injective envelope.
\end{corollary}

\section{Gorenstein flat and Gorenstein projective precovers for complexes}

We recall the following\\
 \begin{definition}
A module $M$ is Gorenstein projective if there is an exact and $Hom
(-, Proj)$ exact complex $ \ldots \rightarrow P_1 \rightarrow P_0
\rightarrow P_{-1} \rightarrow P_{-2}\rightarrow \ldots $ of
projective modules such that $M = Ker(P_0 \rightarrow P_{-1})$.
\end{definition}



The Gorenstein projective complexes are defined in a similar manner, but working with resolutions of complexes.\\
 We recall that for two complexes $X$ and $Y$, $Hom(X,Y)$ denotes
the group of morphisms of complexes from $X$ to $Y$.\\
We also recall that a complex $P$ is projective if the functor
$Hom(P,-)$ is exact. Equivalently, $P$ is projective if and only if
$P$ is exact and for each $n \in Z$, $Ker (P_n \rightarrow P_{n-1})$
is a projective module. For example, if $M$ is a projective module,
then the complex
$$ \ldots \rightarrow 0 \rightarrow M \xrightarrow{Id} M \rightarrow
0 \rightarrow \ldots$$ is projective. In fact, any projective
complex is uniquely up to isomorphism the direct sum of such
complexes (one such complex for each $n \in Z$).\\

 By \cite{garcia:99:covers}, a complex $D$ is called
Gorenstein projective if there exists an exact sequence of complexes
$$ \ldots \rightarrow P_1 \rightarrow P_0
\rightarrow P_{-1} \rightarrow P_{-2}\rightarrow \ldots $$ such that

1) for each $i \in Z$, $P_i$ is a projective complex\\
2) $Ker (P_0 \rightarrow P_{-1}) = D$\\
3) the sequence remains exact when $Hom(-,P)$ is applied to it for
any projective complex $P$.\\

We prove in this section the existence of Gorenstein flat covers
of complexes over two sided noetherian rings, and the existence of \textbf{special}
Gorenstein projective precovers of complexes over commutative noetherian rings
of finite Krull dimension. This result generalizes and improves Theorem 4.26 of
\cite{murfet:10:totally} in two directions: on one side it is
established for the category $Ch(R)$ of unbounded complexes, and on the other
hand we prove that our Gorenstein precover is {\it special}. This property has
been shown to be crucial in defining the cofibrant and fibrant replacements in
(abelian) model category structures on $Ch(R)$ (see
\cite{gillespie:07:kaplansky}). We
would like to stress that our methods are necessarily different from those of
\cite{murfet:10:totally}.

We begin by proving the existence of Gorenstein flat precovers and covers over two sided noetherian rings.\\

\begin{proposition}
Let $R$ be a two sided noetherian ring. The class of Gorenstein flat complexes is covering in $Ch(R)$.
\end{proposition}

\begin{proof}
By \cite{enochs:02:kaplansky} Proposition 2.10, the class of
Gorenstein flat modules is Kaplansky and closed under direct limits. Then by \cite{stovicek:11:deconstructibility} this class is deconstructible. By \cite{stovicek:11:deconstructibility} Proposition 4.3, the class of complexes of Gorenstein flat modules is deconstructible, so it is precovering. But over a  two sided noetherian ring a complex is Gorenstein flat if and only if it is a complex of
Gorenstein flat modules (\cite{enochs:11:gorenstein complexes}, Lemma 12 and
Lemma 13). So the class of Gorenstein flat complexes
is precovering. This class of complexes is also closed under direct limits, so it is covering.
\end{proof}

We consider next the question of the existence of Gorenstein projective precovers for complexes.\\
For modules, Enochs and Jenda showed that when $R$ is a Gorenstein ring the class of Gorenstein projective modules is precovering. Then J{\o}rgensen  showed the existence of Gorenstein projective precovers over commutative noetherian rings with dualizing complexes. Recently, Murfet and Salarian extended his result to commutative noetherian rings of finite Krull dimension.

Their goal in \cite{murfet:10:totally} was to introduce a triangulated category of totally acyclic
complexes of flat modules, which plays the role of $K_{tac}(ProjR)$ for any noetherian ring $R$ (in fact they work in a more general setting, that of complexes of flat sheaves over noetherian schemes).\\
To accomplish this they started with a construction developed by Neeman, who
defined $N(Flat)$ as the Verdier quotient $\frac{K(Flat)}{K_{pac}(Flat)}$,
with $K(Flat)$ the homotopy category of complexes of flat modules, and
$K_{pac}(Flat)$ the full subcategory of pure acyclic complexes in $K(Flat)$ (it
is known that a complex of flat modules is pure acyclic if and only if it is a
flat complex in the sense of Garc\'{\i}a Rozas' definition from [21]). Then they
considered the full subcategory of $N(Flat)$, $N_{tac}(Flat)$, of N-totally
acyclic complexes of flat modules (i.e. exact and $Inj \otimes -$ exact
complexes of flat modules).  Their results in \cite{murfet:10:totally} indicate
that this is the "correct" triangulated category
one can use in order to generalize aspects of
Gorenstein homological algebra to schemes. 

We show that when $R$ is noetherian the class of N-totally acyclic
complexes of flat modules is precovering.

In the following we use $\widetilde{GorFlat}$ to denote the class of exact
complexes $F$ with $Z_n(F) \in \mathcal{GF}$ for each $n$. Since the class of
Gorenstein flat modules is Kaplansky and is also closed under direct limits,
extensions and retracts, the class $\widetilde{GorFlat}$ is covering in $Ch(R)$
(by \cite[Theorem 4.12]{gillespie:07:kaplansky}, or see also
\cite[Corollary 2.11]{enochs:02:kaplansky} and \cite[Corollary 3.1]{EEI}).
By \cite{gillespie:07:kaplansky}, its right orthogonal class,
$\widetilde{GorFlat}^\bot$, consists of the complexes $X$ with each $X_n \in
\mathcal{GF}^\bot$ and such that for any $G \in \widetilde{GorFlat}$, every $u \in
Hom(G,X)$ is homotopic to zero.

\begin{proposition}
Let $R$ be a noetherian ring. Then the class of N-totally acyclic complexes of
flat modules is precovering in $Ch(R)$.
\end{proposition}

\begin{proof}
- Let $P$ be a complex of flat $R$-modules. Since the class of
$\widetilde{GorFlat}$ complexes is covering, there is an exact
sequence
$$0 \rightarrow K \rightarrow F \rightarrow P \rightarrow 0$$ with
$F \in \widetilde{GorFlat}$ and with $K \in
\widetilde{GorFlat}^\perp $. In
particular, each module $K_n$ is in $\mathcal{GF}^\perp$.\\
For each n we have an exact sequence $$0 \rightarrow K_n \rightarrow
F_n \rightarrow P_n \rightarrow 0$$ Since $P_n$ is flat and $K_n \in
\mathcal{GF}^\perp$, the sequence is split exact. So $K_n$ is a direct
summand of $F_n$, so it is Gorenstein flat. But then $K_n \in
\mathcal{GF} \bigcap \mathcal{GF}^\perp$ gives that $K_n$ is flat for each n.
Therefore $F_n$ is flat for each n. So the complex $F$ is N-totally
acyclic. Also, for each N-totally acyclic complex $D$ we have that
$D$ is in
$\widetilde{GorFlat}$, so $Ext^1(D,K)=0$.\\

- Let $X$ be any complex of $R$-modules. Since the class of exact complexes of flat modules, $dw(Flat) \bigcap
\mathcal{E}$, is precovering (\cite{enochs:11:cotorsion:pairs}, example 2), there is an exact sequence $$0
\rightarrow H \rightarrow P  \rightarrow X \rightarrow 0$$ with $P$
exact complex of flat modules and with $H$ in $(dw(Flat)
\bigcap \mathcal{E})^\perp$. \\
By the above there is an exact sequence $$0 \rightarrow K
\rightarrow F \rightarrow P \rightarrow 0$$ with $F$ an N-totally
acyclic complex of flat modules and $K \in N_{tac} (Flat)^\perp$.\\
We form the commutative diagram

\[
\begin{diagram}
\node{}\node{K}\arrow{s}\arrow{e,=}\node{K}\arrow{s}\\
\node{0}\arrow{e}\node{M}\arrow{s}\arrow{e}\node{F}\arrow{s}\arrow{e}\node{X}\arrow{s,=}\arrow{e}\node{0}\\
\node{0}\arrow{e}\node{H}\arrow{e}\node{P}\arrow{e}\node{X}\arrow{e}\node{0}
\end{diagram}
\]

So we have an exact sequence $$0 \rightarrow M \rightarrow F
\rightarrow X \rightarrow 0$$ with $F$ N-totally acyclic complex of
flat modules. Both $K$ and $H$ are in $N_{tac}(Flat)^\perp$, so $M$
also satisfies $Ext^1(D,M) = 0$ for any N-totally acyclic complex
$D$.
\end{proof}

We recall that over a commutative noetherian ring of finite Krull dimension $d$ every Gorenstein flat module $M$ has finite Gorenstein projective
dimension, and $G.p.d _R(M) \le d$.

We prove now the existence of \textbf{special} Gorenstein projective precovers in $Ch(R)$ over a commutative noetherian ring $R$ of finite Krull dimension.

The proof uses the fact that over such a ring $R$, a complex is Gorenstein projective if and only if it is a complex of Gorenstein projective $R$-modules (\cite{enochs:11:gorenstein complexes}, Theorem 3).

\begin{proposition}
If $R$ is commutative noetherian of finite Krull dimension, then every complex $X$ of $R$-modules has a special Gorenstein projective precover.
\end{proposition}

\begin{proof}

Let $dim$ $R=d$.\\
- We show first that every Gorenstein flat complex $G$ has a special Gorenstein projective precover.\\

Let $$0 \rightarrow \overline{G} \rightarrow P_{d-1} \rightarrow \ldots \rightarrow P_0 \rightarrow G \rightarrow 0$$

be a partial projective resolution of $G$. Then for each $j$ we have an exact sequence of modules
$$0 \rightarrow \overline{G}_j \rightarrow P_{d-1,j} \rightarrow \ldots \rightarrow P_{0,j} \rightarrow G_j \rightarrow 0$$

Since $Gpd$ $G_j \le d$ it follows that each $\overline{G}_j$ is Gorenstein projective. Thus $\overline{G}$ is a Gorenstein projective complex (by \cite{enochs:11:gorenstein complexes}, Theorem 3). So $\overline{G}$ has an exact and $Hom(-,Proj)$ exact complex of projective complexes $$0 \rightarrow \overline{G} \rightarrow T_{d-1} \rightarrow \ldots \rightarrow T_0 \rightarrow \ldots$$\\
Let $T=Ker(T_{-1} \rightarrow T_{-2})$. Then $T$ is a Gorenstein projective complex, and we have a commutative diagram:\\

\[
\begin{diagram}
\node{0}\arrow{e}\node{\overline{G}}\arrow{s,=}\arrow{e}\node{T_{d-1}}\arrow{s}\arrow{e}\node{\cdots}\arrow{e}\node{T_1}\arrow{s}\arrow{e}\node{T_0}\arrow{s}\arrow{e}\node{T}\arrow{s}\arrow{e}\node{0}\\
\node{0}\arrow{e}\node{\overline{G}}\arrow{e}\node{P_{d-1}}\arrow{e}\node{\cdots}\arrow{e}\node{P_1}\arrow{e}\node{P_0}\arrow{e}\node{G}\arrow{e}\node{0}
\end{diagram}
\]

Therefore we have an exact sequence:\\
$$0 \rightarrow T_{d-1} \rightarrow P_{d-1} \oplus T_{d-2} \rightarrow \ldots \rightarrow P_1 \oplus T_0 \rightarrow P_0 \oplus T \xrightarrow{\delta} G \rightarrow 0$$

Let $V=Ker{\delta}$. Then $V$ has finite projective dimension, so $Ext^1(W,V)=0$
for any Gorenstein projective complex $W$.\\
We have an exact sequence $0 \rightarrow V \rightarrow P_0 \oplus T \rightarrow G \rightarrow 0$ with $P_0 \oplus T$ Gorenstein projective and with $V$ of finite projective dimension. Thus $P_0 \oplus T \rightarrow G$ is a special Gorenstein projective precover.\\

- We prove now that every complex $X$ has a special Gorenstein projective precover.\\

Let $X$ be any complex of $R$-modules. By Proposition 10, $X$ has an exact sequence $$0 \rightarrow Y \rightarrow G \rightarrow X \rightarrow 0$$ with $G$ Gorenstein flat and with $Ext^1(U,Y)=0$ for any Gorenstein flat complex $U$.\\

By the above, there is an exact sequence $$0 \rightarrow L \rightarrow P \rightarrow G \rightarrow 0$$ with $P$ Gorenstein projective and with $L$ complex of finite projective dimension.\\
Form the pullback diagram

\[
\begin{diagram}
\node{}\node{L}\arrow{s}\arrow{e,=}\node{L}\arrow{s}\\
\node{0}\arrow{e}\node{M}\arrow{s}\arrow{e}\node{P}\arrow{s}\arrow{e}\node{X}\arrow{s,=}\arrow{e}\node{0}\\
\node{0}\arrow{e}\node{Y}\arrow{e}\node{G}\arrow{e}\node{X}\arrow{e}\node{0}
\end{diagram}
\]

Since $L \in {GorProj}^{\perp}$ and $Y \in GorFlat^{\perp}$ and the sequence $0 \rightarrow L \rightarrow M \rightarrow Y \rightarrow 0$ is exact, it follows that $M \in {GorProj}^{\perp}$.\\
So $0 \rightarrow M \rightarrow P \rightarrow X \rightarrow 0$ is exact with $P$ Gorenstein projective and with $M \in {GorProj}^{\perp}$

\end{proof}

\bibliographystyle{plain}

\end{document}